\documentclass[a4paper,reqno,11pt]{amsart}
\usepackage[a4paper,margin=2.5cm]{geometry}
\usepackage[T1]{fontenc}
\usepackage[utf8]{inputenc}
\usepackage[english]{babel}
\usepackage{amssymb,amsmath,amsthm,mathtools}
\usepackage[dvipsnames,svgnames,table]{xcolor}
\usepackage{graphicx}
\usepackage{caption,subcaption}
\usepackage{lmodern}
\usepackage[foot]{amsaddr}
\usepackage[shortlabels]{enumitem}
\usepackage[unicode=true]{hyperref}
\usepackage[capitalise, noabbrev]{cleveref}
\usepackage{thmtools, thm-restate}
\usepackage{comment}
\usepackage{mdframed}
\usepackage{algorithmicx}
\usepackage[noend]{algpseudocode}

\DeclarePairedDelimiter\set{\{}{\}}

\DeclarePairedDelimiter\ceil{\lceil}{\rceil}

\newcommand{\q}[1]{``#1''}
\newcommand{\defin}[1]{\emph{\textcolor{ForestGreen}{#1}}}
\newcommand{\defmath}[1]{\defin{\text{$#1$}}}

\newcommand{\Oh}{\mathcal{O}}

\newcommand{\calL}{\mathcal{L}}

\DeclareMathOperator{\Max}{Max}
\DeclareMathOperator{\Min}{Min}

\let\le\leqslant
\let\ge\geqslant
\let\leq\leqslant
\let\geq\geqslant

\let\subset\subseteq

\let\epsilon\varepsilon

\let\setminus\backslash

\crefformat{equation}{#2(#1)#3}
\crefformat{subsection}{Subsection #2#1#3}

\makeatletter
\def\thm@space@setup{
  \thm@preskip=4mm
  \thm@postskip=0mm
}
\makeatother

\algdef{SE}[DOWHILE]{Do}{doWhile}{\algorithmicdo}[1]{\algorithmicwhile\ #1}

\mdfdefinestyle{dontsplit}{
  hidealllines=true,
  nobreak=true,
  leftmargin=0pt,
  rightmargin=0pt,
  innerleftmargin=0pt,
  innerrightmargin=0pt,
}

\newmdtheoremenv[style=dontsplit]{theorem}{Theorem}
\newmdtheoremenv[style=dontsplit]{lemma}[theorem]{Lemma}
\newmdtheoremenv[style=dontsplit]{observation}[theorem]{Observation}
\newmdtheoremenv[style=dontsplit]{proposition}[theorem]{Proposition}
\newmdtheoremenv[style=dontsplit]{question}[theorem]{Question} 
\newmdtheoremenv[style=dontsplit]{corollary}[theorem]{Corollary} 
\newmdtheoremenv[style=dontsplit]{problem}[theorem]{Problem}
\newmdtheoremenv[style=dontsplit]{conjecture}[theorem]{Conjecture}

\newtheorem*{problem*}{Problem}
\newtheorem*{conjecture*}{Conjecture}
\newtheorem*{remark*}{Remark}
\newtheorem*{question*}{Question} 

\theoremstyle{remark}
\newmdtheoremenv[style=dontsplit]{claim}{Claim}[theorem]

\crefname{claim}{Claim}{Claims}

\newtheorem*{claim*}{Claim}

\newenvironment{proofclaim}[1][]
	{\vspace{-\topsep}\begin{proof}[Proof of the claim] }{\end{proof}}

\graphicspath{{figs/}}

\setlist[enumerate,1]{label=\textup{(\roman*)}}

\makeatletter
\newcommand{\myitem}[1]{%
\item[#1]\protected@edef\@currentlabel{#1}%
}
\makeatother

\hypersetup{
    colorlinks,
    linkcolor={RoyalBlue},
    citecolor={RubineRed},
    urlcolor={blue!80!black},
    pdftitle={On dimension of posets excluding a standard example},
    pdfauthor={Jędrzej Hodor, Piotr Micek, Tom Trotter}
}

\DeclareMathOperator{\idim}{d}
\DeclareMathOperator{\boxi}{box}
\DeclareMathOperator{\mysplit}{split}
\DeclareMathOperator{\PG}{PG}
\newcommand{\har}{\mathrm{H}}

\newcommand{\F}{\mathbb{F}}

\begin{document}

\title[On dimension of posets excluding a standard example]
{On dimension of posets excluding a standard example}

\author[J.~Hodor]{J\k{e}drzej Hodor}
\address[J.~Hodor]{Theoretical Computer Science Department,
Faculty of Mathematics and Computer Science and Doctoral School of Exact
and Natural Sciences, Jagiellonian University, Krak\'ow, Poland}
\email{jedrzej.hodor@gmail.com}

\author[P.~Micek]{Piotr Micek}
\address[P.~Micek]{Theoretical Computer Science Department,
Faculty of Mathematics and Computer Science,
Jagiellonian University, Krak\'ow, Poland}
\email{piotr.micek@uj.edu.pl}

\author[W.~T.~Trotter]{William T. Trotter}
\address[Trotter]{School of Mathematics, Georgia Institute of Technology, Atlanta, Georgia 30332}
\email{trotter@math.gatech.edu}

\thanks{J.~Hodor and P.~Micek were partially supported by the National Science Centre, Poland, under grant no.\ 2022/47/B/ST6/02837}

\date{}

\begin{abstract}
We prove that for every integer $d \geq 3$, there exists $\alpha <1$ such that the maximum dimension of an $n$-element poset containing no standard example of order $d$ as an induced subposet is $\Oh_d(n^\alpha)$. 
For $d=3$, we obtain an upper bound $\Oh(n^{2/3}\log n)$, while we also construct a family of $S_3$-free posets with dimension $\Omega(n^{1/3}/\log n)$.
The proofs were found by OpenAI's ChatGPT-6 Astra.
\end{abstract}

\maketitle

\section{Introduction}\label{sec:introduction}

The \defin{dimension} of a poset $P$, denoted \defin{$\dim(P)$}, is the least positive integer $d$ 
such that
$P$ is isomorphic to a subposet of $\mathbb{R}^d$ equipped with the product order.\footnote{In the product order of $\mathbb{R}^d$, for $(x_1,\dots,x_d),(y_1,\dots,y_d) \in \mathbb{R}^d$, we have $(x_1,\dots,x_d) \leq (y_1,\dots,y_d)$ if and only if $x_i \leq y_i$ for every $i \in [d]$.} 
For an integer $d\ge2$, the \defin{standard example} of 
\defin{order} $d$, denoted by \defin{$S_d$}, consists of
minimal elements $a_1,\ldots,a_d$ and maximal elements $b_1,\ldots,b_d$,
with $a_i<b_j$ for $i,j \in [d]$ if and only if $i\ne j$. 
See \cref{fig:se}.

\begin{figure}[ht]
  \centering
  \includegraphics{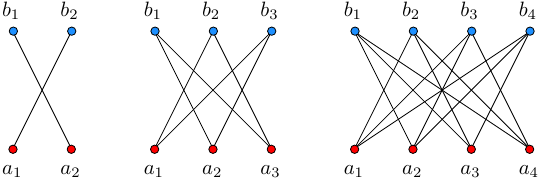}
  \caption{
    The standard examples of order $2$, $3$, and $4$.
    }
   \label{fig:se}
\end{figure}

Dimension is arguably the most widely studied measure of complexity for posets, introduced by Dushnik and Miller~\cite{DM41}.
Standard examples in posets are the canonical structure forcing dimension to be large. 
Indeed, $\dim(S_d)=d$ and since dimension is a monotone parameter, 
if $P$ contains $S_d$ as an induced subposet,  
then $\dim(P)\geq d$. 
However, standard examples are not the only way to drive dimension up; 
there exist infinite families of posets with no large standard examples and arbitrarily large dimension.
Among others: incidence posets of complete graphs (as proved by Dushnik and Miller~\cite{DM41}), interval orders (as proved by Bogart, Rabinovich, and Trotter~\cite{Bogart76}), and adjacency posets of graphs of large girth and large chromatic number (as proved by Felsner and Trotter~\cite{FT00}).

We study dimension of posets that exclude a fixed standard example as an induced subposet.
Accordingly, for an integer $d \geq 2$, we say that a poset $P$ is \defin{$S_d$-free} when $P$ does not contain $S_d$ as an induced subposet.
For all integers $d \geq 2$ and $n \geq 1$, let
\[
\defmath{f(d,n)}=\max\{\dim(P):|P|=n\text{ and }P\text{ is }S_d\text{-free}\}.
\]
Posets that are $S_2$-free are exactly interval orders, as proved by Fishburn~\cite{Fishburn}.
Combining the results of Füredi, Hajnal, Rödl, and Trotter~\cite{FHRT91}, and Hoşten and Morris~\cite{HM} (see detailed explanation in \cite[Section 1.2]{Bir2015}) we obtain that the maximum dimension of an $n$-element interval order is $\log\log n + \frac{1}{2} \log\log\log n + \Oh(1)$.
This gives a very precise estimate of $f(2,n)$.

Bir\'o, Hamburger, and P\'or~\cite{BHP15} proved that for every integer $d \geq 3$, we have $f(d,n) = o(n)$. 
For example, for $d=3$ their proof gives $f(3,n)=\mathcal{O}\left(n\frac{\log\log\log n}{\log\log n}\right)$. 
On the other hand, Bir\'o, Hamburger, Kierstead, P\'or, Trotter, and Wang~\cite[Theorem~6.2]{BHKPTW20} proved that, for every integer $d\ge3$, there exists $\beta_d > 0$ with $\lim_{d\to\infty}\beta_d=1$ such that $f(d,n) = \Omega(n^{\beta_d})$.
We provide a qualitatively matching upper bound.

\begin{theorem}\label{thm:upper}
    For every integer $d\ge3$, there is $\alpha < 1$ such that $f(d,n)\le \Oh( n^{\alpha})$.
\end{theorem}

Since for every fixed integer $d \geq 3$, the function $f(d,n)$ grows polynomially, the next natural problem is to estimate the exponent in this polynomial growth.
For every integer $d \geq 3$, we define
\[\defmath{\gamma_d} = \inf\{\gamma \in [0,1] : \text{$f(d,n) < C \cdot n^\gamma$ for every integer $n \geq 1$ and some real constant $C$}\}.\]

The proof of \cref{thm:upper} gives that 
$\gamma_d \leq 1 - \frac{1}{2^{d-1}-1}$ for every integer $d \geq 3$.
Bir\'o et al.~\cite{BHKPTW20} used random posets of height two, followed by deleting elements to eliminate standard examples to show that $\gamma_d \geq 1 - \frac{2d-1}{d(d-1)}$ for every integer $d \geq 3$. 
Using well-known algebraic constructions, we slightly improve the lower bound.

\begin{theorem}\label{thm:lower}
We have
\[\gamma_3 \geq \frac{1}{3} \ \ \text{ and } \ \ \gamma_d \geq 1 - \frac{1}{d-2} \text{ for every integer $d \geq 4$}.\]
\end{theorem}

Determining the value of $\gamma_3$ remains an exciting open problem. 
Currently, we only know that $1/3 \leq \gamma_3 \leq 2/3$.

\section{Preliminaries}\label{sec:preliminaries}

All posets in this paper are finite.
For a positive real number $x$, we denote by \defin{$\log x$} the logarithm of $x$ of base $2$ and by \defin{$\ln x$}, the natural logarithm of $x$, i.e.\ of base $e$. 
For a positive integer $d$, we denote by \defin{$[d]$} the set $\{1,\dots,d\}$, and $\defmath{[0]} = \emptyset$.
For a set $A$, we denote by \defin{$\calL_A$} the set of all linear orders on $A$.

All graphs in this paper are simple and undirected.
For a graph $G$, we denote by \defin{$V(G)$} its set of vertices and by \defin{$E(G)$} its set of edges.
A \defin{complement} of a graph $G$ is the graph $\defmath{\overline{G}} = (V(G),\binom{V(G)}{2}\setminus E(G))$.
A graph $G$ is \defin{$k$-regular} for a positive integer $k$ if every vertex of $G$ is of degree $k$.

Let $P$ be a poset.
The \defin{height} of $P$ is the maximum size of a chain in $P$. 
The set of minimal elements of $P$ is denoted by \defin{$\Min(P)$} and the set of maximal elements of $P$ is denoted by \defin{$\Max(P)$}.
A \defin{min-max pair} in $P$ is a pair of elements $(a,b)$ such that $a \in \Min(P)$ and $b \in \Max(P)$.
An \defin{isolated element} of $P$ is an element incomparable to all other elements in $P$.
Given a subset $X$ of elements of $P$, we denote with \defin{$P[X]$} the subposet of $P$ induced by the elements in $X$.
For an element $x$ of $P$, the \defin{downset} of $x$ in $P$ is the set $\defmath{D_P(x)} = \{y \in P : y < x \text{ in } P\}$, and the \defin{upset} of $x$ in $P$ is the set $\defmath{U_P(x)} = \{y \in P : x < y \text{ in } P\}$.
For two subsets $X$ and $Y$ of the elements of $P$, we write \defin{$X < Y$} in $P$ when $x < y$ in $P$ for all $x \in X$ and $y \in Y$.
When $Y = \{y\}$, we simply write \q{$X < y$ in $P$}.
Given a linear extension $L$ of a subposet of $P$ and an incomparable pair $(a,b)$ in $P$ such that both $a$ and $b$ are in $L$, we say that $(a,b)$ is \defin{reversed} in $L$ if $b < a$ in $L$.

A linear order $L$ on the elements of $P$ is a \defin{linear extension} of $P$ when for all elements $x$ and $y$ of $P$ with $x < y$ in $P$, we have $x < y$ in $L$.
A \defin{realizer} of a poset $P$ is a collection $\calL$ of linear extensions of $P$ such that every incomparable pair in $P$ is reversed in some order in $\calL$.
A standard combinatorial definition of poset dimension (equivalent with the one presented in the introduction) is that the dimension of $P$ is the minimum positive number of linear extensions that form a realizer of $P$.
See \cite{T92} for more details.



\section{The upper bound}

\subsection{Reducing to height two posets: the split of a poset}\label{ssec:splits}
We now discuss an operation that takes an arbitrary poset and outputs a poset of height $2$ with essentially the same dimension and standard example number. 
This operation was first studied by Kimble~\cite{Kim73}. 
The \defin{split} of a poset $P$ is the poset $\defmath{\mysplit(P)}$ obtained from $P$ as follows.
For every element $x$ in $P$, we create two elements of $\mysplit(P)$: 
a lower copy $x^-$ and 
an upper copy $x^+$. 
The relation in $\mysplit(P)$ is given by
\[
x^-<y^+\ \textrm{in $\mysplit(P)$}\quad\Longleftrightarrow\quad x\le y\ \textrm{in $P$},
\]
for all $x$ and $y$ in $P$.
Note that $\mysplit(P)$ has height $2$ and if $P$ has $n$ elements then $\mysplit(P)$ has $2n$ elements.
Also, $\mysplit(P)$ has no isolated elements. 
It turns out that dimension of $\mysplit(P)$ may differ only by one from the dimension of $P$, and furthermore, $\mysplit(P)$ preserves $S_d$-freeness for $d \geq 3$.
The proof of the first fact can be found in the last author's book~\cite[Theorem~5.2]{T92} and the second in a paper of Bir\'o, Hamburger, and P\'or~\cite[Lemma 4.3]{BHP15}.
Both facts admit short and elementary proofs.

\begin{lemma}[\cite{BHP15,T92}]\label{lem:split}
For every poset $P$, we have
\begin{enumerate}
\item $\dim(P)\leq \dim(\mysplit(P))$;
\item \label{item:split-Sd}
if $P$ is $S_d$-free, 
then $\mysplit(P)$ is $S_d$-free, 
for every integer $d\ge3$.
\end{enumerate}
\end{lemma}

\subsection{Translating to a fractional variant}
Throughout most of this section, we work with posets of height $2$ and no isolated elements.
Let $P$ be such a poset.
We define two auxiliary parameters measuring complexity of a subset of incomparable pairs in $P$.
They correspond to \emph{interval dimension} and \emph{fractional interval dimension}.
Interval dimension was introduced by Trotter and Bogart~\cite{Trotter76} and the fractional version of it was introduced by Biró, Hamburger, Pór, and Trotter~\cite[Section~4]{Bir2015}.

Let $P$ be a poset of height $2$ with no isolated elements. 
We set $A = \Min(P)$.
Let $I$ be a subset of min-max incomparable pairs in $P$.
We define \defin{$\idim(P,I)$} as the minimum positive integer $d$ for which there exist $L_1,\dots,L_d \in \calL_{A}$ such that for every $(a,b) \in I$, there exists $i \in [d]$ with $D_P(b) < a$ in $L_i$. 
Next, we define the fractional variant of the parameter.
Let \defin{$\idim^*(P,I)$} be the minimum\footnote{Note that \q{minimum} here is justified. First, without loss of generality, we may assume that the weight function takes values in $[0,1]$. Then, the feasible set of weight functions can be written as $[0,1]^{\calL_{A}} \cap \bigcap_{(a,b) \in I} \{w : \sum_{L \in \calL_{A}, \ D_P(b) <a \text{ in } L} w(L) \geq 1\} \cap \{w : \sum_{L \in \calL_{A}} w(L) \geq 1\}$, which is an intersection of a compact set and finitely many closed sets, hence a compact set. It follows that the continuous objective $w \mapsto \sum_{L \in \calL_A} w(L)$ attains the minimum.} real number $d \geq 1$ such that there exists a function $w\colon \calL_{A} \to [0,\infty)$ such that 
\begin{enumerate}
    \item $\sum \{w(L) : L \in \calL_A\}  = d$,
    \item $\sum \{w(L) : L \in \calL_A \text{ and } D_P(b)< a \text{ in } L\} \geq 1$ for every $(a,b) \in I$.
\end{enumerate}

Note that by the definitions,
\[\idim^*(P,I) \leq \idim(P,I) \leq |A|.\]

Note also that if $I_1$, $I_2$, and $I$ are subsets of min-max incomparable pairs of $P$ with $I\subseteq I_1\cup I_2$ then
\begin{equation}
\idim^*(P,I) \leq \idim^*(P,I_1) + \idim^*(P,I_2).
\label{eq:d*-union-bound}
\end{equation}

It is well-known that the value $\idim(P,I)$ is closely related to $\dim(P)$ when $I$ is the set of all min-max incomparable pairs in $P$.
See e.g.\ \cite[Proposition~4]{Bir2015}.
We reprove this relation for completeness.

\begin{lemma}\label{lem:dimension}
Let $P$ be a poset of height $2$ with no isolated elements. 
Let $I$ be the set of all min-max incomparable pairs in $P$. 
Then,
\[
\dim(P)\le\idim(P,I)+1.
\]
\end{lemma}
\begin{proof}
Let $A = \Min(P)$, $B = \Max(P)$, and $d=\idim(P,I)$.
Fix $L_1,\ldots,L_d\in \calL_A$ witnessing the definition of $\idim(P,I)$. 
For each $i\in[d]$, let $M_i$ be a linear extension of $P$ obtained from $L_i$ by putting each $b\in B$ as low as possible in that linear order. 
Let $(a,b) \in I$. 
By the choice of $L_1,\ldots,L_d$, 
there is $i\in[d]$ such that 
$D_P(b) < a$ in $L_i$. 
Therefore, we have $D_P(b) < b < a$ in $M_i$.
Thus, $(a,b)$ is reversed in $M_i$.

We construct one more linear extension $M_0$ of $P$ as follows: all elements of $A$ are below all elements of $B$ in $M_0$; elements of $A$ are ordered in a reversed order to $M_1$; elements of $B$ are ordered in a reversed order to $M_1$. 
(Note that we need here that $d\geq1$.)
The linear extensions $M_0,M_1,\dots,M_d$ form a realizer of $P$.
\end{proof}

The values $\idim(P,I)$ and $\idim^*(P,I)$ are related in the following way.

\begin{lemma}\label{lem:round}
Let $P$ be a poset of height $2$ with no isolated elements.
Let $I$ be a subset of min-max incomparable pairs in $P$.
If $I \neq \emptyset$, then
\[
\idim(P,I)\le\ceil{\idim^*(P,I) \cdot \max\{1,\ln|I|\}}.
\]
\end{lemma}
\begin{proof}
Let $A = \Min(P)$ and $t = \idim^*(P,I)$.
    There exists a function $w\colon \calL_A \to [0,\infty)$ such that (1) $\sum \{w(L) : L \in \calL_A\} = t$ and (2) $\sum \{ w(L) : L \in \calL_A \text{ and } D_P(b) < a \text{ in } L\}\geq 1$ for every $(a,b) \in I$.
    Now, consider a probability distribution on $\calL_A$ such that the probability of drawing $L \in \calL_A$ is equal to $w(L) / t$.
    By the condition (2), we have that for every $(a,b) \in I$, we have $\mathbb{P}(D_P(b) < a \text{ in a sampled order}) \geq 1/t$.

    Sample $r = \ceil{t \cdot \max\{1,\ln|I|\}}$ independent orders $\calL_A'$ from $\calL_A$ with the given probability distribution. 
    Note that $r\geq1$. 
    The probability that there exists $(a,b) \in I$ with $D_P(b) \not< a$ in $L$ for every $L \in \calL_A'$ is at most 
    \[
    |I| \cdot (1-1/t)^r < |I|\cdot e^{-r/t}
    \le \frac{|I|}{\max\{e,|I|\}} \leq 1.
    \]
 Thus, with non-zero probability, $\calL_A'$ witnesses that $\idim(P,I) \leq r$. 
\end{proof}

\begin{lemma}\label{lem:prefix}
Let $P$ be a poset of height $2$ with no isolated elements. 
Let $k$ be a positive integer.
Set $A = \Min(P)$ and $B = \Max(P)$.
Let $L_0 \in\calL_A$. 
Let $J$ be a set of min-max incomparable pairs in $P$ such that
if $(a,b)\in J$ then $|\set{a'\in A :\ \textrm{$a'<b$ in $P$ and $a<a'$ in $L_0$}}|< k$.
%
Then,
\[
\idim^*(P,J) \leq e\cdot k.
\]
\end{lemma}
\begin{proof}
First, suppose that $k=1$. 
Then $D_P(b) < a$ in $L_0$ for each $(a,b)\in J$. 
Let $w:\calL_A\to[0,\infty)$ such that 
$w(L_0)=1$ and $w(L)=0$ for all $L\neq L_0$. 
Note that $w$ witnesses that $\idim^*(P,J)=1$. 
Now suppose that $k\geq2$. 
Independently mark each $a \in A$ with probability $p=1/k$.
Consider a linear order $L_0'$ on $A$ defined in the following way.
All unmarked elements of $A$ precede all marked elements of $A$, the linear order on unmarked elements is inherited from $L_0$, and 
the linear order on marked elements is also inherited from $L_0$.

Consider $(a,b) \in J$.
Let $X_{(a,b)} = \{a' \in A:\textrm{$a'<b$ in $P$ and $a<a'$ in $L_0$}\}$.
Note that if element $a$ is marked and all the elements of $X_{(a,b)}$ are unmarked, then $D_P(b) < a$ in $L_0'$.
Thus, since $|X_{(a,b)}| < k$, we have
\begin{align}\label{eq:prob:L0'}
\mathbb{P}(D_P(b) < a \text{ in } L_0') \geq p(1-p)^{|X_{(a,b)}|} \geq \frac{1}{k} \cdot \left(1-\frac{1}{k}\right)^{k-1} \geq \frac{1}{ke}.
\end{align}
Recall that $L_0'$ is a random variable.
For every $L \in \calL_A$, we set $w(L) = \mathbb{P}(L_0' = L) \cdot ek$.
Then, $\sum \{ w(L) : L \in \calL_A\} = ek$ and by \cref{eq:prob:L0'}, for every $(a,b) \in J$, $\sum\{ w(L) : L \in \calL_A \text{ and } D_P(b) < a \text{ in } L \} \geq 1$.
It follows that $\idim^*(P,J) \leq ek$.
\end{proof}


\subsection{Excluding a standard example}

In this subsection, we present two basic facts on $S_d$-free posets of height $2$.

\begin{lemma}\label{lem:structure-1}
Let $P$ be a poset of height $2$ with no isolated elements.
Let $X \subset \Min(P)$ and $Y \subset \Max(P)$.
Suppose that $P[X \cup Y]$ is $S_d$-free for some integer $d \geq 2$.
Then there exists $Y'\subseteq Y$ with $|Y'|<d$ such that 
\[\bigcup_{b \in Y} X \setminus D_P(b) = \bigcup_{b \in Y'} X \setminus D_P(b).\]
\end{lemma}

\begin{proof}
Let $Y' \subset Y$ be inclusion-wise minimal such that
$\bigcup_{b \in Y} X \setminus D_P(b) = \bigcup_{b \in Y'} X \setminus D_P(b)$.
We claim that $|Y'| < d$.
The minimality implies that for every $b \in Y'$, there exists $x_b \in X$ such that 
$(x_b,b)$ is an incomparable pair in $P$ and
$x_b < y'$ in $P$ for all $y'\in Y'\setminus\set{b}$. 
Therefore, if $|Y'| \geq 2$, then
the set $\{x_b : b \in Y'\}\cup Y'$ induces a standard example of order $|Y'|$ in $P$. 
Thus, since $P[X\cup Y]$ is $S_d$-free, we obtain $|Y'| < d$, as claimed.
\end{proof}

\begin{lemma}\label{lem:structure-2}
Let $P$ be a poset of height $2$. 
Let $(u,v)$ be a min-max incomparable pair in $P$. 
Suppose that $P$ is $S_d$-free for some integer $d \geq 3$.
Then $P[D_P(v)\cup U_P(u)]$ is $S_{d-1}$-free.
\end{lemma}
\begin{proof}
Note that for every $S \subset D_P(v) \cup U_P(u)$ that induces a standard example of order $k$ in $P$, the set $S \cup \{u,v\}$ induces a standard example of order $k+1$ in $P$; see \cref{fig:build-sd}.
\end{proof}

\begin{figure}[tp]
  \centering
  \includegraphics{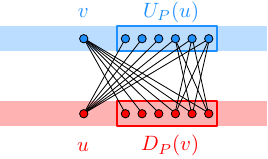}
  \caption{
        An illustration of the proof of \cref{lem:structure-2}.
    }
   \label{fig:build-sd}
\end{figure}

\subsection{Constructing a realizer}

We are ready to state and prove our main technical statement.

\begin{lemma}\label{lem:ind}
Let $d$ be an integer with $d\geq2$. 
Let $P$ be a poset of height $2$ with no isolated elements. 
Let $A \subset \Min(P)$ and $B \subset \Max(P)$.
Let $I$ be a set of incomparable pairs in $P$ with $I \subset A \times B$.
Suppose that $P[A \cup B]$ is $S_d$-free. 
If $A \neq \emptyset$, then
\[
\idim^*(P,I)\leq4ed\cdot |A|^{\alpha_d}
\quad\text{where}\quad
\alpha_d=1-\frac{1}{2^{d-1}-1}.
\]
\end{lemma}

\begin{figure}[tp]
  \centering
  \includegraphics{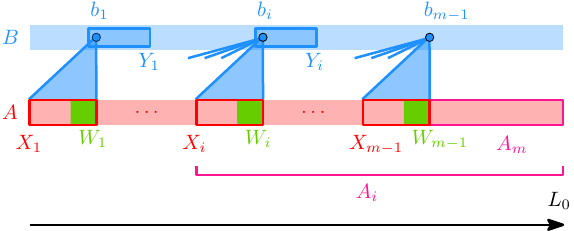}
  \caption{
        An illustration of objects considered in the proof of \cref{lem:ind}.
    }
   \label{fig:proof}
\end{figure}

\begin{proof}
We proceed by induction on $d$.
Assume that $A \neq \emptyset$.
When $d=2$, then since $P[A \cup B]$ is $S_2$-free, the sets $D_P(b) \cap A$ for $b\in B$ are linearly ordered by inclusion. 
Choose an order $L_A$ on $A$ in which every $D_P(b)\cap A$ is an initial segment. 
Let $L_0$ be a linear order on $\Min(P)$ such that $\Min(P)\setminus A <  A$ and $A$ is ordered by $L_A$. 
Note that $L_0$ witnesses that $\idim^*(P,I)=1$, as required.

Next, assume that $d\geq3$.
Set $t=\ceil{2|A|^{\alpha_d}}$. 
We construct a sequence $(A_1,\dots,A_m)$ of subsets of $A$ as follows.
Let $A_1=A$.
Suppose that we have already constructed $A_i$ for some positive integer $i$.
If $|D_P(b)\cap A_i|<t$ for every $b\in B$, then we set $m=i$ and end the construction.
Otherwise, among the elements $b\in B$ with $|D_P(b)\cap A_i|\geq t$, choose $b_i$ minimizing $|D_P(b_i)\cap A_i|$ and set
\[
X_i=D_P(b_i)\cap A_i \ \ \text{ and } \ \
A_{i+1}=A_i\setminus X_i.
\]
Since $|X_i|\geq t$ at every step, the construction terminates.
The sets $X_1,\dots,X_{m-1},A_m$ are pairwise disjoint and their union is $A$; the final set $A_m$ may be empty.
See \cref{fig:proof}.

For every $i\in[m-1]$, let
\[
Y_i=\{b\in B:|D_P(b)\cap A_{i+1}|<t\leq |D_P(b)\cap A_i|\}.
\]
For every $i\in[m-1]$ and $b\in Y_i$, the choice of $b_i$ gives
\[|X_i|\leq |D_P(b)\cap A_i|=|D_P(b)\cap X_i|+|D_P(b)\cap A_{i+1}|,\]
and  it follows that
\begin{equation}
|X_i\setminus D_P(b)|
=|X_i|-|D_P(b)\cap X_i|
\leq |D_P(b)\cap A_{i+1}|
<t.
\label{eq:small-Xi-Db}
\end{equation}
For every $i\in[m-1]$, let
\[
W_i=\bigcup_{b\in Y_i}(X_i\setminus D_P(b)).
\]
Since $P[A\cup B]$ is $S_d$-free, by \cref{lem:structure-1}, there exists 
$Y_i'\subseteq Y_i$ with $|Y_i'|\leq d-1$ such that 
$W_i=\bigcup_{b\in Y_i'}(X_i\setminus D_P(b))$. 
This together with~\cref{eq:small-Xi-Db} gives
\begin{equation}
|W_i| \leq (d-1)(t-1).
\label{eq:Wi}
\end{equation}
Let $L_0\in\calL_{\Min(P)}$ be a linear order with
$\Min(P) \setminus A < X_1<\dots<X_{m-1}<A_m$, in which $X_i\setminus W_i<W_i$ for every $i\in[m-1]$.
All remaining choices of order in $L_0$ are arbitrary.
Set $k=d(t-1)$ and define
\begin{align*}
J&=\{(a,b)\in I:|\{x\in D_P(b):a<x\text{ in }L_0\}|< k\},\\
I_i&=\{(a,b)\in I:a\in X_i \text{ and } |D_P(b)\cap A_{i+1}|\geq t\}
\quad\text{for every }i\in[m-1].
\end{align*}
We claim that
\begin{equation}\label{eq:ind-pairs}
I=J\cup\bigcup_{i\in[m-1]}I_i.
\end{equation} 
Obviously, $J\cup\bigcup_{i\in[m-1]}I_i \subset I$. 
In order to prove the opposite containment consider $(a,b)\in I$. 
Suppose first that $a\in A_m$, then every element of $D_P(b)$ after $a$ in $L_0$ belongs to $A_m$.
By the stopping condition, $|D_P(b)\cap A_m|\leq t-1<k$, so $(a,b)\in J$.
Now suppose that $a\in X_i$ for some $i\in[m-1]$.
If $|D_P(b)\cap A_{i+1}|\geq t$, then $(a,b)\in I_i$.
If $|D_P(b)\cap A_i|<t$, then since every element of $D_P(b)$ after $a$ in $L_0$ belongs to $A_i$, we again have that 
$|\set{x\in D_P(b): a<x \textrm{ in $L_0$}}|\leq t-1< k$, 
so again $(a,b)\in J$.
Thus, it remains to deal with the case that $|D_P(b)\cap A_{i+1}|<t\leq |D_P(b)\cap A_{i}|$. In this case $b\in Y_i$ and since $(a,b)$ is an incomparable pair, $a\in W_i$. 
By the choice of $L_0$, every element of $D_P(b)$ after $a$ belongs to $W_i\cup A_{i+1}$.
Thus,
\begin{align*}
|\{x\in D_P(b):a<x\text{ in }L_0\}|
&\leq |D_P(b)\cap (W_i\setminus\{a\}\cup  A_{i+1})|\\
&\leq |W_i \setminus \{a\}|+|D_P(b)\cap A_{i+1}|\\
&<(d-1)(t-1)+(t-1)=k&&\textrm{by~\cref{eq:Wi}.}
\end{align*}
Thus, $(a,b)\in J$.
This completes the proof of \cref{eq:ind-pairs}.

Fix $i\in[m-1]$ and suppose that 
$I_i\neq\emptyset$. 
Automatically, we obtain that $|A_{i+1}|\geq t$. 
For every $u\in A_{i+1}$, let
\[
I_{i,u}=\{(a,b)\in I_i:u<b\text{ in }P\}.
\]
Since $u\in A_{i+1} = A_i \setminus D_P(b_i)$, we have that $(u,b_i)$ is an incomparable pair in $P$. 
Thus, Lemma~\ref{lem:structure-2} gives that 
$P[(D_P(b_i)\cap A)\cup (U_P(u)\cap B)]$ is $S_{d-1}$-free. 
In particular, $P[X_i\cup (U_P(u)\cap B)]$ is $S_{d-1}$-free. 
Apply induction to $X_i$, $U_P(u)\cap B$, and $I_{i,u}$ to obtain 
\begin{align}\label{d*:iu}
    \idim^*(P,I_{i,u})\leq 4e(d-1)|X_i|^{\alpha_{d-1}}.
\end{align}
Let $w_{i,u}\colon \calL_{\Min(P)}\to[0,\infty)$ be a weight function witnessing that, i.e., 
\begin{enumerate}[label=(iu\arabic*)]
    \item \label{item:iu1} $\sum\set{w_{i,u}(L) : L \in \calL_{\Min(P)}} = \idim^*(P,I_{i,u})$ and 
    \item \label{item:iu2} 
$\sum\set{w_{i,u}(L) : L \in \calL_{\Min(P)} \text{ and } D_P(b)< a \text{ in } L} \geq 1$ for each $(a,b) \in I_{i,u}$.
\end{enumerate}
Let \[w_i = \frac{1}{t}\cdot\sum_{u\in A_{i+1}} w_{i,u}.\]
Recall that we assumed $I_i\neq\emptyset$ and therefore we have $|A_{i+1}|\geq t$. 
In particular, we have $\sum\set{w_i(L) \mid L\in\calL_{\Min(P)}}\geq \frac{1}{t}\cdot |A_{i+1}|\geq1$.
We claim that $w_i$ witnesses 
\begin{align}\label{eq:d*-i}
    \idim^*(P,I_i) \leq \frac{1}{t} \cdot |A_{i+1}| \cdot 4e(d-1)|X_i|^{\alpha_{d-1}}.
\end{align}
In other words, we must prove that
\begin{enumerate}[label=(i\arabic*)]
    \item \label{item:i1} $\sum\set{w_{i}(L) : L \in \calL_{\Min(P)}} \leq 1/t \cdot |A_{i+1}| \cdot 4e(d-1)|X_i|^{\alpha_{d-1}}$ and 
    \item \label{item:i2} 
    $\sum\set{w_{i}(L) : L \in \calL_{\Min(P)} \text{ and } D_P(b)< a \text{ in } L} \geq 1$ for each $(a,b) \in I_{i}$.
\end{enumerate}
Item \ref{item:i1} follows, as
\begin{align*}
\sum\set{w_i(L): L\in\calL_{\Min(P)}} 
&= \frac{1}{t}\cdot\sum_{u\in A_{i+1}}\sum\set{w_{i,u}(L): L\in\calL_{\Min(P)}}\\
&= \frac{1}{t}\cdot\sum_{u\in A_{i+1}}\idim^*(P,I_{i,u})&&\textrm{by \ref{item:iu1}}\\
&\leq \frac{1}{t}\cdot |A_{i+1}|\cdot 4e(d-1)|X_i|^{\alpha_{d-1}}&&\textrm{by \cref{d*:iu}.}
\end{align*}
We now prove \ref{item:i2}.
Recall that for each $(a,b)\in I_i$, 
we have $|D_P(b)\cap A_{i+1}|\geq t$. Thus, for at least $t$ choices of $u\in A_{i+1}$ we have $(a,b)\in I_{i,u}$. Therefore, for each $(a,b)\in I_{i}$
\begin{align*}
&\sum\set{w_{i}(L) : L\in\calL_{\Min(P)} \text{ and } D_P(b) < a 
\text{ in } L}\\
&\quad= \frac{1}{t}\sum_{u\in A_{i+1}}\sum\set{w_{i,u}(L) : L\in\calL_{\Min(P)} \text{ and } D_P(b) < a 
\text{ in } L}\\
&\quad\geq \frac{1}{t}\sum_{u\in D_P(b)\cap A_{i+1}}\sum\set{w_{i,u}(L): L\in\calL_{\Min(P)} \text{ and } D_P(b) < a 
\text{ in } L}\\
&\quad\geq \frac{1}{t}\sum_{u\in D_P(b)\cap A_{i+1}} 1  \ \geq \frac{1}{t}\cdot |D_P(b)\cap A_{i+1}|\geq1&& \text{ by \ref{item:iu2}.}
\end{align*}

Since $\alpha_{d-1}<1$, $|X_i|\geq t$, and the sets $X_i$ are pairwise disjoint, we have
\begin{align*}
    \sum_{i\in[m-1]}|A_{i+1}| |X_i|^{\alpha_{d-1}}
    &\leq |A| \sum_{i\in[m-1]}(|X_i|^{\alpha_{d-1}-1}\cdot |X_i|)\\
    &\leq |A| t^{\alpha_{d-1}-1} \cdot \sum_{i\in[m-1]}|X_i|\\
    &\leq |A|^2t^{\alpha_{d-1}-1}.
\end{align*}

Recall that by~\cref{eq:ind-pairs} we have $I=J\cup\bigcup_{i\in[m-1]}I_i$. 
By \cref{lem:prefix}, we have $\idim^*(P,J)\leq ek$.
Altogether, we obtain
\begin{align*}
\idim^*(P,I)
&\leq \idim^*(P,J) + \sum_{i\in[m-1] \text{ and } I_i \neq \emptyset} \idim^*(P,I_i) && \text{by \cref{eq:d*-union-bound}}\\
&\leq ek + \sum_{i\in[m-1]} \frac{|A_{i+1}|}{t}\cdot 4e(d-1)|X_i|^{\alpha_{d-1}}&& \text{by \cref{lem:prefix} and \cref{eq:d*-i}}\\
&\leq e\cdot d(t-1) + 4e(d-1) |A|^2t^{\alpha_{d-1}-2}.
\end{align*}
The choice of $\alpha_d$ is supposed to balance out the two terms: $t$ and $|A|^2 t^{\alpha_{d-1} - 2}$ in the above estimate.
Note that they are supposed to be roughly equal to $|A|^{\alpha_d}$.
It follows that $\alpha_d = 2/(3-\alpha_{d-1})$.
Recall that $\alpha_2 = 0$.
This simple recurrence relation can be solved to $\alpha_d = 1 - 1 / (2^{d-1}-1)$ as in the statement of the lemma.

By the definition of $t$, we have $2|A|^{\alpha_d}\leq t<2|A|^{\alpha_d}+1$.
Consequently,
\[
e\cdot d(t-1)\leq e\cdot 2d|A|^{\alpha_d}.
\]
Moreover, $\alpha_d(3-\alpha_{d-1})=2$ and $\alpha_{d-1}<1$, so
\[
4e(d-1)|A|^2t^{\alpha_{d-1}-2}
\leq4e(d-1)2^{\alpha_{d-1}-2}|A|^{2+\alpha_d(\alpha_{d-1}-2)}
\leq2ed|A|^{\alpha_d}.
\]
We used the fact that $2^{\alpha_{d-1}-1} < 1$.
Substituting these estimates, we obtain
\[
\idim^*(P,I)\leq 4ed\,|A|^{\alpha_d}.
\]
This completes the proof.
\end{proof}

\begin{proof}[Proof of \cref{thm:upper}]
Let $d \geq 3$ and $n \geq 2$ be integers.
Let $P$ be an $n$-element $S_d$-free poset and let $Q=\mysplit(P)$.
Note that $Q$ is of height $2$ and has no isolated elements. 
By \cref{lem:split}, the poset $Q$ is $S_d$-free.
Let $I$ be the set of all min-max incomparable pairs in $Q$. 
If $I=\emptyset$, then $\dim(P)\leq\dim(Q)\leq\idim(Q,I)+1=2$. 
Thus, assume that $I\neq\emptyset$. 
Observe that $|\Min(Q)| = n$ and $|I| \leq n^2$.
The application of \cref{lem:ind} below is to $Q$, $I$, $\Min(Q)$, and $\Max(Q)$.
Additionally, $\alpha_d$ is as in the statement of \cref{lem:ind}.
We have
\begin{align*}
\dim(P)
&\leq \dim(Q) && \text{by \cref{lem:split}}\\
&\leq \idim(Q,I)+1 && \text{by \cref{lem:dimension}}\\
&\leq \lceil\idim^*(Q,I)\max\set{1,\ln|I|}\rceil+1 && \text{by \cref{lem:round}}\\
&\leq\idim^*(Q,I)\cdot\ln(n^2)+2\\
&\leq8edn^{\alpha_d}\ln n+2 && \text{by \cref{lem:ind}}.
\end{align*}
Since $\alpha_d = 1 - 1/(2^{d-1}-1)<1$, choosing any $\alpha\in(\alpha_d,1)$ gives $f(d,n)=\Oh_d(n^\alpha)$.
\end{proof}

\section{Lower bounds}\label{sec:lower}
In this section, we prove \cref{thm:lower}, i.e.\ we construct $S_d$-free posets with relatively large dimension.
In \cref{ssec:S3}, we prove the case of $d=3$.
This is done via a simple counting argument using a well-known construction of $C_6$-free bipartite graphs with $n$ vertices and $\Omega(n^{4/3})$ edges.
In \cref{ssec:box-spectral}, we introduce some necessary tools: boxicity, relations between boxicity and dimension, and spectral estimates.
In \cref{ssec:general}, we complete the proof of \cref{thm:lower} by considering incidences in high-dimensional projective spaces.

\subsection{Posets with no \texorpdfstring{$S_3$}{S3} and large dimension}\label{ssec:S3}


The following statement implies the case of $d=3$ in \cref{thm:lower}.

\begin{theorem} For every integer $n \geq 2$, we have
\[
f(3,n) > \frac{n^{1/3}}{8\log n}.
\]
\label{thm:f3n}
\end{theorem}

Let $G$ be a bipartite graph with a fixed bipartition $(A,B)$ of the vertex set.
We denote by \defin{$P_G$} the poset on $A\cup B$ whose strict comparisons are
\[
a<b\text{ in }P_G
\quad\Longleftrightarrow\quad
a\in A,\ b\in B, \textrm{ and }  ab\in E(G).
\]

\begin{lemma}\label{lem:counting-lower}
Let $N$ and $M$ be integers with $N\geq2$ and $M\geq0$. 
Let $G$ be a bipartite graph with a fixed bipartition, $N$ vertices and $M$ edges.
There is a spanning subgraph $H$ of $G$ such that
\[
\dim(P_H)\geq\frac{M}{\log(N!)}.
\]
\end{lemma}
\begin{proof}
There are $2^M$ distinct spanning subgraphs of $G$ and they give rise to a family $\mathcal{P}$ of $2^M$ distinct (labeled) posets.
Let $d$ be the maximum dimension of posets in $\mathcal{P}$. 
Therefore, for each $P \in \mathcal{P}$ we can fix a sequence of exactly 
$d$ linear orders on the $N$-element ground set that uniquely determine $P$. 
Therefore, we have at most $(N!)^d$ posets in $\mathcal{P}$.
Thus, we obtained $2^M\leq (N!)^d$, which proves the assertion.
\end{proof}

Note that the comparability graph of $S_3$ is a cycle of length $6$.
Let $G$ be a bipartite graph that is $C_6$-free, i.e., it does not contain a cycle of length $6$ as a subgraph. 
These graphs are relevant as for each subgraph $H$ of $G$, we have that $P_H$ is $S_3$-free.
We need the following classical result from the extremal graph theory.

\begin{theorem}[Benson~\cite{Ben66}] 
For every prime power $q$, there is a graph $G_q$ such that the following holds. Let $m=(q+1)(q^2+1)$. Then 
\begin{enumerate}
    \item $G_q$ is bipartite and $C_6$-free,
    \item $G_q$ has $2m$ vertices, and
    \item $G_q$ has $m(q+1)$ edges.
\end{enumerate}
\label{thm:C6free}
\end{theorem}

\begin{proof}[Proof of \cref{thm:f3n}]
The statement is trivial for $n < 30$, so assume that $n \geq 30$.
Let $q$ be the largest positive power of $2$ with $2(q+1)(q^2+1)\leq n$. 
Let $m=(q+1)(q^2+1)$.
By the choice of $q$, we have 
\begin{align}\label{eq:n-q}
n < 2(2q+1)((2q)^2+1)\leq 64q^3,\ \textrm{so $q>\frac{n^{1/3}}{4}$.}
\end{align}

Let $G_q$ be the graph given by~\cref{thm:C6free}. By~\cref{lem:counting-lower} there exists a subgraph $H$ of $G_q$ such that
\[
\dim(P_H) \geq \frac{m(q+1)}{\log((2m)!)}.
\]
Recall that for each integer $x\geq1$, we have $\log(x!)=\sum_{i\in[x]}\log i\leq x\log x$. Therefore, we have
\begin{align*}
\dim(P_H)\geq \frac{m(q+1)}{2m\cdot\log(2m)}\geq\frac{q+1}{2\log(2m)}> \frac{n^{1/3}}{8\log n}.
\end{align*}
In the last inequality, we used \cref{eq:n-q} and the fact that $2m \leq n$.
Adding isolated elements to a poset preserves $S_3$-freeness and cannot decrease dimension.
Hence, by adding $n - 2m$ isolated elements to $P_H$, we obtain an $n$-element $S_3$-free poset of dimension larger than $n^{1/3} / (8\log n)$. This completes the proof.
\end{proof}

\subsection{Boxicity and spectral tools}\label{ssec:box-spectral}

In this subsection, we introduce tools needed to prove the \q{$d\geq 4$} part of \cref{thm:lower}.

An \defin{interval graph} is a graph isomorphic to an intersection graph of closed intervals on the real line.
An \defin{interval representation} of an interval graph is an assignment of intervals to the vertices of the graph witnessing the isomorphism.
The \defin{intersection} of two graphs $G_1$ and $G_2$ on the same set of vertices $V$ is the graph $(V,E(G_1)\cap E(G_2))$.
The \defin{boxicity} of a graph $G$, denoted by \defin{$\boxi(G)$}, is the minimum number of interval graphs on $V(G)$ whose intersection is $G$.
For a bipartite graph $G$ with a fixed bipartition, Adiga, Bhowmick, and Chandran~\cite[Corollary~1]{ABC11} proved that
\begin{equation}\label{eq:box-dimension}
\boxi(G)\leq\dim(P_G),
\end{equation}
and $\dim(P_G)\leq2\boxi(G)$ provided that $G$ is not a complete graph.

We will use the bipartite expander mixing lemma, which dates back to Haemers~\cite[Theorem~3.1.1]{Hae79}.
A concise proof of the statement implying the one that we use can be found in the appendix of a paper by Lund and Saraf~\cite[Lemma~8]{Lund2016}.

A \defin{bipartite incidence matrix} of a bipartite graph $G$ with parts $A$ and $B$ is a $|A| \times |B|$ matrix $M$ with rows indexed by $A$, columns indexed by $B$, and $M_{ab}=1$ if $ab\in E(G)$ and $M_{ab}=0$ otherwise.

\begin{lemma}[Bipartite expander mixing lemma] 
Let $k$ and $m$ be integers with $1\leq k<m$.
Let $G$ be a $k$-regular bipartite graph with parts $A$ and $B$ of size $m$. 
Let $\sigma$ be the second largest singular value\footnotemark \ of the bipartite incidence matrix of $G$.
For all $X\subseteq A$ and $Y\subseteq B$, we have
\[
\left||E(G) \cap (X \times Y)|-\frac{k}{m}|X||Y|\right|
\leq\sigma\sqrt{|X||Y|}.
\]\label{lem:bipartite-mixing}
\end{lemma}
\footnotetext{
In some of the statements in the literature, second largest \emph{eigenvalue} of the symmetric adjacency matrix is used instead of the second largest \emph{singular value} of the bipartite incidence matrix.
However, in this case, these two coincide.
Let $M$ be a bipartite incidence matrix of a bipartite graph $G$.
Its singular values are the nonnegative square roots of the eigenvalues of $MM^{\mathsf T}$.
The usual symmetric adjacency matrix of $G$ is
\[
M' = \begin{pmatrix}0&M\\M^{\mathsf T}&0\end{pmatrix},
\]
and its eigenvalues are the singular values of $M$ and their negatives, counted with multiplicity.
Thus, $\sigma$ is also the second largest eigenvalue of $M'$, ordered by value rather than absolute value.
}

The next lemma is a key connection between dimension of a poset of height $2$ and spectral properties of its cover graph.
For every positive integer $m$, let
    \[\defmath{\har_m} = \sum_{i=1}^m \frac{1}{i}.\]
Given a graph $G$ and a subset of vertices $X$, the \defin{antineighbourhood} of $X$ in $G$ is the set $\defmath{A_G(X)} = \{v \in V(G)\setminus X : \text{$uv \notin E(G)$ for every $u \in X$}\}$.

\begin{lemma}\label{lem:spectral-lower}
Let $k$ and $m$ be integers with $1\leq k<m$.
Let $G$ be a $k$-regular bipartite graph with parts $A$ and $B$ of size $m$.
Let $\sigma$ be the second largest singular value of the bipartite incidence matrix of $G$.
Then,
\[
\dim(P_G)\geq \frac{1}{\sigma^2} \cdot 
\frac{k^2(m-k)}{2m \cdot \har_m}.
\]
\end{lemma}
\begin{proof}
Let $G^+$ be obtained from $G$ by making each of $A$ and $B$ a clique. 
Let $G'$ be a bipartite complement of $G$, i.e., the graph obtained from $G$ by taking a complement and making each of $A$ and $B$ an independent set.
Note that $G' = \overline{G^+}$.
For each $i \in [2m-1]$, let 
\[c_i = \max\{|A_{G^+}(Z)| : Z \subset V(G) \text{ and } |Z| = i\}.\]
Adiga, Chandran, and Sivadasan~\cite[Theorem~2 and Lemma~4]{ACS14} proved the following inequalities that we reprove for completeness.
\begin{claim*}[\cite{ACS14}]
    \[
    2\boxi(G)\geq\boxi(G^+)
    \geq\frac{|E(G')|}{\sum_{i=1}^{2m-1}c_i}.
    \]
\end{claim*}
\begin{proofclaim}
    For the first inequality, set $b = \boxi(G)$ and take an interval representation of each of the $b$ interval graphs $H_1,\dots,H_b$ on $V(G)$ whose intersection is $G$.
    Rescale its intervals to lie in $[0,1]$.
    For every $v \in V(G)$ and every $j \in [b]$, replace the interval $[\ell_v,r_v]$ corresponding to $v$ in the representation of $H_j$ by
    \[
    I_1^j(v)=
    \begin{cases}
    [0,r_v]&\text{if }v\in A,\\
    [\ell_v,1]&\text{if }v\in B,
    \end{cases}
    \qquad
    I_2^j(v)=
    \begin{cases}
    [\ell_v,1]&\text{if }v\in A,\\
    [0,r_v]&\text{if }v\in B.
    \end{cases}
    \]
    We obtain $2b$ interval graphs whose intersection is $G^+$.

    Next, we prove the second inequality.
    Let $H$ be an interval graph on $V(G)$ that is a supergraph of $G^+$.
    Order its vertices $v_1,\dots,v_{2m}$ by nondecreasing right endpoints.
    If for some integers $i,j \in [2m]$ with $i < j$, we have $v_iv_j\notin E(H)$, then the interval of $v_j$ starts after the right endpoint of $v_i$.
    In particular, $v_j \in A_{G^+}(\{v_1,\dots,v_i\})$.
    It follows that for a given $i \in [2m-1]$, there are at most $c_i$ numbers $j \in [2m]$ with $i < j$ and $v_iv_j \notin E(H)$.
    Therefore, $|E(\overline{H})| \leq \sum_{i=1}^{2m-1}c_i$.

    Let $b = \boxi(G^+)$ and let $H_1,\dots,H_b$ be interval graphs on $V(G)$ whose intersection is $G^+$.
    Then, $E(\overline{G^+}) \subset E(\overline{H_1}) \cup \dots \cup E(\overline{H_b})$, and in particular, $|E(\overline{G^+})| \leq |E(\overline{H_1})| + \dots + |E(\overline{H_b})|$.
    Since for every $\ell \in [b]$, the graph $H_\ell$ contains $G^+$, we obtain
    \[|E(\overline{G^+})| \leq b \cdot \sum_{i=1}^{2m-1}c_i, \]
    which proves the second inequality.
\end{proofclaim}

If $X\subseteq A$ and $Y\subseteq B$ have no edge between them in $G$, \cref{lem:bipartite-mixing} gives
\[
\frac{k}{m}|X||Y|\leq\sigma\sqrt{|X||Y|}, \quad \text{and hence}\quad
|X||Y|\leq\frac{m^2\sigma^2}{k^2}.
\]
Let $Z \subset V(G)$ with $|Z| = i$ be such that $c_i = |A_{G^+}(Z)|$ for some fixed $i \in [m]$.
We may assume that $Z$ is contained in one of the parts $A$ or $B$.
It follows that $|Z| \cdot |A_{G^+}(Z)| \leq \frac{m^2\sigma^2}{k^2}$.
In particular, $c_i\leq m^2\sigma^2/(k^2i)$ when $i\leq m$, and $c_i=0$ when $i>m$.
Since $|E(G')|=m(m-k)$, applying \cref{eq:box-dimension} and the claim, we obtain
\[
\dim(P_G) \geq \boxi(G)\geq \frac{|E(G')|}{2\cdot \sum_{i=1}^{2m-1} c_i} \geq 
\frac{m(m-k)}{2(m^2\sigma^2/k^2)\har_m} = \frac{1}{\sigma^2} \cdot \frac{k^2(m-k)}{2m \cdot \har_m},
\]
as desired.
\end{proof}

We note that an asymptotic bound as in \cref{thm:f3n} can also be obtained from an application of \cref{lem:spectral-lower}.
An advantage is that this gives an explicit construction omitting the counting argument (\cref{lem:counting-lower}).
The second largest singular value of the bipartite incidence matrix of the graph $G_q$ witnessing the statement of \cref{thm:C6free} is equal to $\sqrt{2q}$, see \cite[Lemma~4.7]{AFNW06}.
Additionally, $G_q$ is $(q+1)$-regular.
Thus, \cref{lem:spectral-lower} applies with $m = (q+1)(q^2+1)$, $k = q+1$, and $\sigma = \sqrt{2q}$.
Then, $\dim(P_{G_q}) = \Omega(q / \log q)$ and $G_q$ has $\Theta(q^3)$ vertices.

\subsection{Projective spaces}\label{ssec:general}

F\"uredi and Kahn~\cite{FK88} proved that the lattice of a projective plane of order $q$ has dimension at least $q/\ln q$, where $\ln$ denotes the natural logarithm.
Deleting its least and greatest elements preserves dimension and leaves the point--line incidence poset.
This poset is $S_4$-free: two distinct points lie on a unique common line, whereas two minimal elements of $S_4$ have two common upper elements.
It has $2(q^2+q+1)$ elements.
Consequently, the bound
\[
f(4,n)=\Omega\!\left(\frac{n^{1/2}}{\log(n+1)}\right)
\]
already follows from their result.

We use the same incidence construction in higher-dimensional projective spaces.
Let $r\geq2$ be an integer and let $q$ be a prime power.
We define a poset \defin{$\PG(r,q)$} as follows.
Consider a vector space $\F_q^{r+1}$.
The one-dimensional subspaces of $\F_q^{r+1}$ are the minimal elements of the poset (we call them \defin{lines}) and the $r$-dimensional subspaces of $\F_q^{r+1}$ are the maximal elements of the poset (we call them \defin{hyperplanes}).
Additionally, for a line $\ell$ and a hyperplane $h$, we set
\[
\ell<h\text{ in }\PG(r,q)
\quad\Longleftrightarrow\quad
\ell\subseteq h.
\]
There are $m = \frac{q^{r+1}-1}{q-1} = 1+q+\cdots+q^r$ lines and the same number of hyperplanes in $\PG(r,q)$.
Also, every element of $\PG(r,q)$ is comparable with $k=\frac{q^{r}-1}{q-1}=1+q+\cdots+q^{r-1}$ elements.
The second largest singular value of the bipartite incidence matrix of the cover graph of $\PG(r,q)$ is equal to $q^{(r-1)/2}$.
These are standard projective-space parameters; see Alon, Moran, and Yehudayoff~\cite[Appendix~D, Lemma~25]{AMY16} for the spectral calculation.

\begin{lemma}\label{prop:projective-lower}
Let $r\geq2$ be an integer and let $q$ be a prime power.
Let $m = 1+q+\cdots+q^r$.
Then, $\PG(r,q)$ is $S_{r+2}$-free and 
\[
\dim(\PG(r,q))\geq\frac{q^{r-1}}{4 \cdot \har_m}.
\]
\end{lemma}
\begin{proof}
First, we prove that $\PG(r,q)$ is $S_{r+2}$-free.
Suppose to the contrary that lines $\ell_1,\dots,\ell_{r+2}$ and hyperplanes $h_1,\dots,h_{r+2}$ induce $S_{r+2}$ in $\PG(r,q)$.
For every $i\in[r+2]$, choose a nonzero vector $v_i\in \ell_i$.
Note that for every $i \in [r+2]$, $\{v_j : j \in [r+2]\setminus\{i\}\} \subset h_i$.
Since $v_i \notin h_i$, it follows that $v_i \notin \text{span}\{v_j : j \in [r+2]\setminus\{i\}\}$.
In particular, $v_1,\dots,v_{r+2}$ are linearly independent.
This is a contradiction as $\F_q^{r+1}$ is of dimension $r+1$.
Therefore, $\PG(r,q)$ is indeed $S_{r+2}$-free.

We apply \cref{lem:spectral-lower} to the cover graph of $\PG(r,q)$ with $k = 1 + q + \dots + q^{r-1}$ and $\sigma = q^{(r-1)/2}$.
Note that $m = qk+1$, $k < m/2$, and $k \geq q^{r-1}$.
We obtain
\[
\dim(\PG(r,q))
\geq \frac{1}{q^{r-1}} \cdot \frac{k^2(m-k)}{2m \cdot\har_m}
\geq \frac{1}{q^{r-1}} \cdot \frac{k^2}{4 \cdot\har_m}
\geq\frac{q^{r-1}}{4\cdot \har_m},
\]
as required.
\end{proof}

\begin{proof}[Proof of \cref{thm:lower}]
    Let $d\geq 4$ be an integer.
    Set $r = d-2$.
For a sufficiently large $n$, choose the largest integer $q \geq 2$ that is a power of $2$ and $2(1+q+\cdots+q^r)\leq n$.
Then,
\[
n<2\sum_{i=0}^r(2q)^i\leq 2^{r+2}q^r,
\]
so $q=\Omega_r(n^{1/r})$.
Also, $\har_m=\Oh(\log(n+1))$.
By \cref{prop:projective-lower}, adding isolated elements to $\PG(r,q)$ gives an $n$-element $S_d$-free poset with dimension
\[
\Omega_d\!\left(\frac{n^{(r-1)/r}}{\log(n+1)}\right)
=\Omega_d\!\left(\frac{n^{1-1/(d-2)}}{\log(n+1)}\right).
\]
This proves the remaining part of \cref{thm:lower}.
\end{proof}






\section*{AI statement}
The proofs were found by OpenAI's ChatGPT-6 Astra. 
The authors digested the initially generated output and made an effort to present the argument in a readable and optimized form. The authors take responsibility for
the mathematical correctness of the presented arguments.

\bibliographystyle{plain}
\bibliography{bibliography}
\end{document}